\documentclass[11pt]{amsart}
\usepackage{etex}
\usepackage{amsmath, amscd, amsthm, amssymb, graphics, mathrsfs, setspace,fancyhdr,geometry,tabularx,shapepar, hyperref, xcolor, tikz, cite, booktabs}
\usepackage[all]{xy}
\hypersetup{colorlinks=false}
\theoremstyle{plain}
\newtheorem{theorem}{Theorem}[section]

\newtheorem{lemma}[theorem]{Lemma}
\newtheorem{corollary}[theorem]{Corollary}
\newtheorem{proposition}[theorem]{Proposition}

\newtheorem{conjecture}[theorem]{Conjecture}
\newtheorem*{theorem*}{Theorem}
\newtheorem*{problem*}{Problem}

\theoremstyle{definition}
\newtheorem{definition}[theorem]{Definition}

\newtheorem{remark}[theorem]{Remark}

\numberwithin{equation}{section}

\renewcommand{\O}{{\mathcal O}}

\newcommand{\Z}{{\mathbb Z}}
\newcommand{\R}{{\mathbb R}}

\newcommand{\gm}{\mathbb{G}_{m}}

\newcommand{\op}[1]{\operatorname{#1}}

\begin{document}

\title[The toric Frobenius morphism and a conjecture of Orlov]{The toric Frobenius morphism and a conjecture of Orlov}

\author[Ballard]{Matthew R Ballard}
\address{Department of Mathematics, University of South Carolina, 
Columbia, SC 29208}
\email{ballard@math.sc.edu}
\urladdr{\url{http://people.math.sc.edu/ballard/}}
\thanks{The first author was partially supported by NSF DMS-1501813. He would also like to thank the Institute for Advanced Study for providing a wonderful research environment. These ideas were developed during his membership.}

\author[Duncan]{Alexander Duncan}
\address{Department of Mathematics, University of South Carolina, 
Columbia, SC 29208}
\email{duncan@math.sc.edu}
\urladdr{\url{http://people.math.sc.edu/duncan/}}
\thanks{The second author was partially supported by NSA grant
H98230-16-1-0309.}

\author[McFaddin]{Patrick K. McFaddin}
\address{Department of Mathematics, University of South Carolina, 
Columbia, SC 29208}
\email{pkmcfaddin@gmail.com}
\urladdr{\url{http://mcfaddin.github.io/}}
\thanks{}

\begin{abstract}
We combine the Bondal-Uehara method for producing exceptional
collections on toric varieties
with a result of the first author and Favero to expand the
set of varieties satisfying Orlov's Conjecture on derived dimension. 
\end{abstract}

\maketitle
\addtocounter{section}{0}



\section{Introduction}



To solve a longstanding question originating with work of Auslander, Rouquier used a new invariant of a triangulated category $\mathsf T$ \cite{RouAus}. This invariant is a measure of the minimal homological complexity of the category $\mathsf T$. He focused particular attention on the case where $\mathsf T = \op{D}^b(\op{coh} X)$ for a scheme $X$ of finite type over a field $k$. In this case, he showed his invariant is always at least the Krull dimension of $X$ with equality in certain situations, like Grassmannians $\op{Gr}(r,n)$. Orlov then asked if Rouquier's invariant, henceforth known as the \textsf{Rouquier dimension}, is actually equal to the Krull dimension in the case $X$ is smooth and projective. He showed the answer was yes in the case of curves \cite{Orlov}. 

Despite being a simply stated question, Orlov's Conjecture seems difficult to address in general. Indeed, supporting evidence comes from individual constructions specialized to particular examples \cite{BF,BFK,Rou,Orlov}. For toric varieties, we give a more robust method based on an idea of Bondal \cite{Bondal} refined by Uehara \cite{Uehara} utilizing the toric Frobenius morphism. The main result asserts that if the Bondal-Uehara method produces a tilting bundle then Orlov's Conjecture holds. We close with applications that illustrate the potency of this simple idea. 

\section{Toric Frobenius and generation time} 


\subsection{Generation time and the Rouquier dimension} \label{section: background gen time}

Let $\mathsf{T}$ be a triangulated category. Recall that this guarantees that for any map $f: A \to B$ in $\mathsf{T}$ there is a triangle
\begin{displaymath}
 A \overset{f}{\to} B \to C(f) \to A[1].
\end{displaymath}
Generally, the assignment $f \mapsto C(f)$ is only well-defined up to an isomorphism of $C(f)$. Even so one commonly calls $C(f)$ the \textsf{cone} over $f$. In \cite{Rou}, Rouquier, building on work of \cite{BV}, introduced a notion of dimension of a triangle category. This notion measures the homological complexity of the category by, roughly, counting cones. Let us be a bit more precise.

\begin{definition} \label{definition: operations on tricat}
If $\mathsf{S}$ is a full subcategory of $\mathsf{T}$, let $\langle \mathsf{S} \rangle$ denote the full subcategory of $\mathsf{T}$ containing $\mathsf{S}$ and which is closed under finite coproducts, summands, and translations. 
 
 Let $\mathsf{S}_1$ and $\mathsf{S}_2$ be full subcategories of $\mathsf{T}$.
Let $\mathsf{S}_1 \ast \mathsf{S}_2$ be the full subcategory of $\mathsf{T}$ consisting of objects $A$ such that there exists a triangle
 \begin{displaymath}
  S_1 \to A \to S_2 \to S_1[1]
 \end{displaymath}
 with $S_i$ an object of $\mathsf{S}_i$.
Define $\mathsf{S}_1 \diamond \mathsf{S}_2 := \langle \mathsf{S}_1 \ast
\mathsf{S}_2 \rangle$.
 
 One inductively defines
 \begin{align*}
  \langle \mathsf{S} \rangle_0 & := \langle \mathsf{S} \rangle \\
  \langle \mathsf{S} \rangle_{n+1} & := \langle \mathsf{S} \rangle_n
\diamond \langle \mathsf{S} \rangle.
 \end{align*}
 
 One says that $\mathsf{S}$ \textsf{generates} $\mathsf{T}$ if any object of $\mathsf{T}$ is isomorphic to an object of $\langle \mathsf{S} \rangle_n$ for some $n$, possibly depending on the object. One says that $\mathsf{S}$ \textsf{strongly generates} $\mathsf{T}$ if there exists an $n$ such that the inclusion $\langle \mathsf{S} \rangle_n \to \mathsf{T}$ is an equivalence. One says that the \textsf{generation time} of $\mathsf{S}$ is the minimal $n$ such that $\langle \mathsf{S} \rangle_n \to \mathsf{T}$ is an equivalence.
 
 If $\mathsf{S}$ consists of a single object $S$, then one also says that $S$ (strongly) generates if $\mathsf{S}$ does. The generation time of the object $S$ is the generation time of $\mathsf{S}$.
\end{definition}

\begin{definition}
 The \textsf{Rouquier dimension} of $\mathsf{T}$ is
 \begin{displaymath}
  \op{rdim} \mathsf{T} := \op{min} \lbrace n \mid \exists \text{ an object } S \text{ so that } \langle S \rangle_n \cong \mathsf{T} \rbrace
 \end{displaymath}
 with the notation $\infty$ used if the set is empty. In other words,
the Rouquier dimension is the minimal generation time among any of the
objects. For a $k$-scheme $X$, we denote $\op{rdim} \text{D}^b (\op{coh} X)$ by $\op{rdim } X$.
\end{definition}

In \cite{Orlov}, Orlov made the following conjecture.

\begin{conjecture} \label{conjecture: Orlov}
 Let $X$ be a smooth quasiprojective variety. Then the Rouquier dimension and the Krull dimension coincide, i.e.,
 \begin{displaymath}
  \op{rdim} X = \op{dim} X.
 \end{displaymath}
\end{conjecture}

Rouquier had already given the lower bound in \cite[Proposition 7.16]{Rou}. 

\begin{proposition} \label{proposition: lower bound}
 For a reduced separated scheme $X$ of finite type over a field, we have $\op{rdim} X \geq \op{dim} X$. 
\end{proposition}

Thanks to this result, verifying Orlov's Conjecture for a given $X$ amounts to finding a particular nice generator whose generation time is $\op{dim} X$. We will also need the following basic property of Rouquier dimension, which follows from \cite[Lemma 3.4]{Rou}. Recall that a functor is dense if every object is isomorphic to a summand of an object in the image. 

\begin{lemma} \label{lemma: essentially}
 Let $F : \mathsf{S} \to \mathsf{T}$ be a dense exact functor of triangulated categories. Then 
 \begin{displaymath}
  \op{rdim} \mathsf{S} \geq \op{rdim} \mathsf{T}.
 \end{displaymath}
\end{lemma}


\subsection{A result on generation time for tilting objects} \label{section: anti-canonical}

Now, we restrict ourselves to a simpler class of generators.  

\begin{definition} \label{definition: tilting object}
 Let $\mathsf{T}$ be a triangulated category.  An object $T$ of $\mathsf{T}$ is called a \textsf{tilting object} if the following two conditions hold:
\begin{enumerate}
\item $\op{Hom}_{\mathsf{T}}(T, T[i]) = 0 \emph{ for all } i \not = 0$;

\item $T$ is a generator for $\mathsf{T}$.
\end{enumerate}
\end{definition}

For these, one can give a more easily computable upper bound on the generation time.
The following is a consequence of \cite[Theorem 3.2]{BF}.

\begin{theorem} \label{theorem: upper bound}
 Let $X$ be a smooth and projective variety. Suppose that $T$ is a tilting object in $\op{D}^b(\op{coh} X)$ and let
 \begin{displaymath}
  m_0(T) := \op{max} \lbrace m \mid \op{Hom}(T,T \otimes_{\mathcal O_X} \omega_X^{\vee}[m]) \not = 0 \rbrace.
 \end{displaymath}
 The generation time of $T$ is bounded above by $\dim X + m_0(T)$. In particular, if there exists a $T$ with $m_0(T) = 0$, then Orlov's Conjecture holds. 
\end{theorem}


\subsection{Toric Frobenius and the anti-nef cone} \label{section:frob}

To search for tilting objects, we follow Bondal \cite{Bondal} and turn
to the toric Frobenius morphism. Let $X$ be a (split) smooth projective 
toric variety of dimension $n$ with fixed torus embedding
$T \hookrightarrow X$ and take $\ell \in \mathbb{N}$.
Define the $\ell^{\text{th}}$ Frobenius map on $T = \gm^n$ to be 
\begin{displaymath}
 (x_1,..., x_n) \mapsto (x_1^{\ell},..., x_n^{\ell}).
\end{displaymath}
This uniquely extends to an endomorphism of $X$ which will be denoted $F_{\ell}$ and called the \textsf{$\ell^{th}$ Frobenius morphism}.
Each sheaf $(F_{\ell})_*(\O_X)$ splits into line bundles.

\begin{definition}
 Let $\op{frob}(X)$ denote the union of all line bundles arising as direct summands of $(F_{\ell})_*(\O_X)$ as $\ell$ varies over $\Z^+$. 
\end{definition}

Thomsen provides an explicit description for $(F_{\ell})_*(\O_X)$ and
shows that the set $\op{frob}(X)$ is finite~\cite[Proposition 6.1]{Thomsen}.
Taking a sufficiently divisible $\ell$, we recover the following:

\begin{proposition} \label{proposition: summands stabilize}
There exists an $\ell$ such that $(F_{\ell})_*(\O_X)$ contains every line
bundle in $\op{frob}(X)$. 
\end{proposition}

Recall that for a normal variety $X$, a Cartier $D$ divisor on $X$ is
\textsf{nef} if $D \cdot C \geq 0$ for every irreducible curve $C
\subset X$. Let $N^1(X)$ be the quotient group of Cartier divisors by
the subgroup of numerically trivial divisors.  The \textsf{nef cone}
$\op{nef}(X)$ is the cone in $N^1(X) \otimes_\Z \R$ given by positive
span of the nef divisors, and the \textsf{anti-nef cone} is the cone
$\op{fen}(X) := -\op{nef}(X) \subset N^1(X) \otimes_\Z \R$.
For smooth projective toric varieties, $\op{Pic}(X) = N^1(X)$.

\begin{definition} 
 We denote the intersection by
 \begin{displaymath}
  \op{bu}(X) := \op{frob}(X) \cap \op{fen}(X) \subset \op{Pic}(X).
 \end{displaymath}
\end{definition}


\subsection{Main result} \label{section:main result}

We can now state the main result. Let 
\begin{displaymath}
 T_{\op{bu}}(X) := \bigoplus_{L \in \op{bu}(X)} L.
\end{displaymath}

\begin{theorem} \label{theorem: main}
 Let $X$ be a smooth projective toric variety. If $T_{\op{bu}}(X)$ is a tilting object and $-K_X$ is nef, then the generation time of $T_{\op{bu}}$ is $\op{dim} X$. In particular, Orlov's Conjecture holds for $X$. 
\end{theorem}

\begin{proof}
 We will apply Theorem~\ref{theorem: upper bound} to $T:= T_{\op{bu}}(X)$. Let's compute $\op{Hom}(T,T \otimes \omega_X^{-1}[m])$ to show that $m_0(T) = 0$.
Note that $T$ is a direct summand of $F_{\ell \ast} \O$ for some $\ell$.
Thus, to get the desired vanishing, we first observe that
 \begin{displaymath}
  \op{Hom}(L, F_{\ell \ast} \O \otimes \omega_X^{-1}[m]) \cong \op{Hom}(L, F_{\ell \ast} (\omega_X^{-\ell})[m])
 \end{displaymath}
 using the projection formula. Using adjunction, we have
 \begin{displaymath}
  \op{Hom}(L, F_{\ell \ast} (\omega_X^{-\ell})[m]) \cong \op{Hom}(F_{\ell}^\ast L, \omega_X^{-\ell}[m]) \cong \op{Hom}(L^\ell, \omega_X^{-\ell}[m]) \cong \op{H}^m(X, (L \otimes \omega_X)^{-\ell}).
 \end{displaymath}
 Since $(L \otimes \omega_X)^{-\ell}$ is nef, its higher cohomology vanishes. 
\end{proof}

\begin{remark}
 One sees immediately from \cite[Theorem 2.1]{BF} that this also computes the global dimension of the finite dimensional algebra $A = \op{End}_X(T)$ as $\op{dim} X$. 
\end{remark}

The next result shows that Orlov's Conjecture can propogate, in $K$-negative ways, to other birational models, or chambers in the secondary fan. 
Following the argument of \cite[Proposition 5.2.5, Corollary 5.2.6]{BFK},
we see that:

\begin{proposition} \label{proposition: enough to get nef Fano DM}
 If Orlov's Conjecture holds for a smooth projective nef-Fano toric DM
stack $X$ that is isomorphic in codimension $\geq
1$ to a smooth projective toric DM stack $Y$, then Orlov's Conjecture holds for $Y$.
\end{proposition}

Thus we have the following:

\begin{corollary} \label{corollary: more cases}
 Let $X$ be a smooth projective variety isomorphic in codimension $\geq
1$ to a smooth projective toric nef-Fano $Y$ with $T_{\op{bu}}(Y)$ a tilting object. Then Orlov's Conjecture holds for $X$. 
\end{corollary}


\subsection{Examples} \label{section: examples}

Despite being fairly innocous, we can leverage the results of Section~\ref{section:main result} into a healthy increase of positive examples of Orlov's Conjecture. 

\begin{proposition} \label{proposition: toric fano 3folds}
 Orlov's Conjecture holds for all smooth Fano toric threefolds. 
\end{proposition}

\begin{proof}
 One can now apply Theorem~\ref{theorem: main} thanks to \cite{Uehara} which guarantees that $T_{\op{bu}}(Y)$ is tilting for $Y$ a smooth toric Fano variety of dimension at most $3$. 
\end{proof}

In \cite{Prabhu}, the Bondal-Uehara method is used to exhibit exceptional collections for a subset of smooth toric Fano fourfolds but, in fact, $m_0(T) = 0$ for all $T$ produced by Prabhu-Naik.

\begin{proposition} \label{proposition: toric fano 4folds}
 Orlov's Conjecture holds for all smooth toric Fano fourfolds. 
\end{proposition}

\begin{proof}
 This follows immediately from \cite[Theorem 7.8]{Prabhu} since being a pullback exceptional collection, in the language of \cite[Section 3.2]{BF}, includes the vanishing of $m_0(T)$.  
\end{proof}

\begin{remark}
 We can use Corollary~\ref{corollary: more cases} to expand the list of toric varieties satisfying Orlov's Conjecture to all those coming from $K$-negative birational maps starting from Propositions~\ref{proposition: toric fano 3folds} and Proposition~\ref{proposition: toric fano 4folds}.
\end{remark}

These results hold for all the corresponding arithmetic toric
varieties as well.

\begin{lemma}
 Let $X$ be a smooth, quasi-projective variety over $k$ and $L/k$ a Galois extension. Then Orlov's Conjecture holds for $X$ if and only if it holds for $X_L$.
\end{lemma}

\begin{proof}
 If it holds for $X$, then \cite[Proposition 5.4]{Sosna} says it holds for $X_L$. Conversely, the projection $\pi: X_L \to X$ is a dense functor as $\pi_\ast \pi^\ast E \cong E^{\oplus |G|}$ for $G = \op{Gal}(L/k)$. So this follows from Lemma~\ref{lemma: essentially} and Proposition~\ref{proposition: lower bound}.
\end{proof}



\begin{thebibliography}{9999999}


\bibitem[BF12]{BF}
Ballard, Matthew; Favero, David. {\em Hochschild dimensions of tilting objects}. Int. Math. Res. Not. IMRN 2012, no. 11, 2607--2645.

\bibitem[BFK14]{BFK14}
Ballard, Matthew; Favero, David; Katzarkov, Ludmil. A category of kernels for equivariant factorizations, II: further implications. J. Math. Pures Appl. (9) 102 (2014), no. 4, 702--757.

\bibitem[BFK17]{BFK}
Ballard, M.; Favero, D.; Katzarkov, L. \emph{Variation of Geometric Invariant Theory quotients and derived categories}. J. Reine Angew. Math. (Crelles Journal), 2017. DOI:
https://doi.org/10.1515/crelle-2015-0096.

\bibitem[BvB03]{BV}
Bondal, A.; van den Bergh, M. {\em Generators and representability of functors in commutative and noncommutative geometry}. Mosc. Math. J. 3 (2003), no. 1, 1--36, 258.

\bibitem[Bon06]{Bondal}
Bondal, A.I. {\em Derived categories of toric varieties}. Oberwolfach Rep., 3(1):284--286, 2006.

\bibitem[Orl09]{Orlov}
Orlov, Dmitri. {\em Remarks on generators and dimensions of triangulated categories}. Mosc. Math. J. 9 (2009), no. 1, 153--159, back matter.

\bibitem[PN17]{Prabhu}
Prabhu-Naik, Nathan. {\em Tilting bundles on toric Fano fourfolds}. J. Algebra 471 (2017), 348--398.

\bibitem[Rou06]{RouAus}
Rouquier, Rapha\"el. {\em Representation dimension of exterior algebras}. Invent. Math. 165 (2006), no. 2, 357--367.

\bibitem[Rou08]{Rou}
Rouquier, Rapha\"el. {\em Dimensions of triangulated categories}. J. K-Theory 1 (2008), no. 2, 193--256. 

\bibitem[Sos14]{Sosna}
Sosna, Pawel. {\em Scalar extensions of triangulated categories}. Appl. Categ. Structures 22 (2014), no. 1, 211--227.

\bibitem[Tho00]{Thomsen}
Thomsen, Jesper Funch. {\em Frobenius direct images of line bundles on toric varieties}. J. Algebra 226 (2000), no. 2, 865--874.

\bibitem[Ueh14]{Uehara}
Uehara, Hokuto. {\em Exceptional collections on toric Fano threefolds and birational geometry}. Internat. J. Math. 25 (2014), no. 7, 1450072, 32 pp.


\end{thebibliography}
\end{document}